\documentclass[reqno]{amsart}
\usepackage{amsmath}
\usepackage{amssymb}
\usepackage{amsthm}
\usepackage{graphicx}
\usepackage{color}
\usepackage{ulem}
\usepackage{epstopdf}
\usepackage{here}

\newtheorem{theorem}{Theorem}[section]

\newtheorem{lemma}[theorem]{Lemma}
\theoremstyle{definition}
\newtheorem{definition}[theorem]{Definition}

\numberwithin{equation}{section}

\newcommand{\mincol}{{\rm mincol}^{\rm Dehn}}

\newcommand{\C}{\mathcal{C}}
\newcommand{\Z}{\mathbb{Z}}

\begin{document}

\title[]{The minimum number of Dehn $\Z$-colors of any nonsplittable $\Z$-colorable link is three}

\author[E.~Matsudo]{Eri Matsudo} 
\address{The Institute of Natural Sciences, Nihon University, 3-25-40 Sakurajosui, Setagaya-ku, Tokyo 156-8550, Japan}
\email{matsudo.eri@nihon-u.ac.jp}

\author[K.~Oshiro]{Kanako Oshiro}
\address{Department of Information and Communication Sciences, Sophia University, Tokyo 102-8554, Japan}
\email{oshirok@sophia.ac.jp}



\date{\today}

\begin{abstract}
Our previous papers \cite{MatsudoOshiroYamagishi-1, MatsudoOshiroYamagishi-2} are the first and second to discuss minimum numbers of “region” colors, while minimum numbers of arc colors such as Fox colors are well-studied.
As the third installment, in this paper, we investigate the minimum number of Dehn $\Z$-colors.
In particular, we show that the minimum number of Dehn $\Z$-colors of a nonsplittable $\mathbb Z$-colorable link is three.
\end{abstract}

\maketitle

\section{Introduction}
Our previous papers \cite{MatsudoOshiroYamagishi-1, MatsudoOshiroYamagishi-2} are the first and second to discuss minimum numbers of “region” colors for knots, while minimum numbers of arc colors such as Fox colors are well-studied, see \cite{AbchirElhamdadiLamsifer,BentoLopes,HanZhou,HararyKauffman, IchiharaMatsudo16,IchiharaMatsudo17,Matsudo,NakamuraNakanishiSatoh13,NakamuraNakanishiSatoh16,Oshiro10,Satoh09,ZhangJinDeng} for example.  
It is remarkable that the minimum number of Dehn $p$-colors exhibits a distinct behavior from that of Fox $p$-colors.
For example, for Fox $p$-colorings with odd primes $p \le 19$, the minimum number of colors is known to be constant regardless of knots.
In contrast, for Dehn $p$-colorings, this value can serve as an invariant that distinguishes  knots for certain primes, such as $p = 7$. 
Therefore, investigating the minimum number of Dehn colors without restricting to knots or to an odd prime $p$ would be a compelling direction for future research.

In this paper, we investigate the minimum number of colors for Dehn $\mathbb{Z}$-colorings of $\mathbb{Z}$-colorable links.

Next theorem is our main result:
\begin{theorem}\label{main_theorem}
The minimum number of Dehn $\Z$-colors of a nonsplittable $\mathbb Z$-colorable link is three.
\end{theorem}

Regarding Fox $\mathbb{Z}$-colorings, it is known from \cite{IchiharaMatsudo17, Matsudo, ZhangJinDeng} that the minimum number of colors for a nonsplittable $\mathbb{Z}$-colorable link is four. Consequently, we obtain an analogous result in which the required number of colors is reduced by one.

This paper is organized as follows:
Section 2 reviews Dehn $\mathbb{Z}$-colorings and their minimum numbers of colors.
In Section 3, by converting the notion of ``simple" from Fox colorings, we provide a definition of simple Dehn $\mathbb{Z}$-colorings and explore their properties.
Finally, Section 4 presents the proof of Theorem~\ref{main_theorem}.

\section{Dehn $\Z$-colorings and minimum numbers of colors}\label{sec:1}

Let $D$ be a diagram of a link $L$ with ${\rm Det}(L)=0$ and $\mathcal{R}(D)$ the set of regions of $D$.
A {\it Dehn $\Z$-coloring} of $D$ is a map $C: \mathcal{R}(D) \to \Z$ 
satisfying the following condition: 
\begin{itemize}
\item for each crossing $c$ with regions 
$x_1, x_2, x_3$, and $x_4$ 
as depicted in Figure~\ref{coloring2},
\[
C(x_1) + C(x_3) =C(x_2) + C(x_4)
\]
holds, where the region $x_2$ is adjacent to $x_1$ by an under-arc and $x_3$ is adjacent to $x_1$ by the over-arc.
\end{itemize}
We call $C(x)$ the {\it color} of a region $x$ by $C$. 
In this paper, as shown on the right-hand side of Figure~\ref{coloring2}, we represent a Dehn $\Z$-coloring $C$ of a knot diagram $D$ by assigning the color $C(x)$ to each region $x$.  
We mean by $(D,C)$ a diagram $D$ given a Dehn $\Z$-coloring $C$, and call it a {\it Dehn $\Z$-colored diagram}. 
We denote by $\mathcal{C}(D, C)$ the set of colors assigned to regions of $D$ by $C$, that is $\mathcal{C}(D, C)={\rm Im}\,C$.
The set of Dehn $\Z$-colorings of $D$ is denoted by ${\rm Col}_{\Z}(D)$.
\begin{figure}[ht]
  \begin{center}
    \includegraphics[clip,width=6cm]{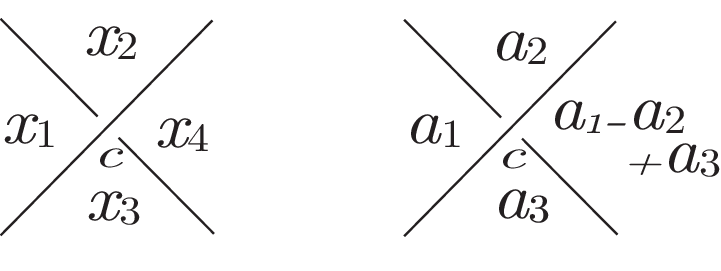}
    \caption{A crossing on $D$ and the one on $(D,C)$ with $C(x_1)=a_1$, $C(x_2)=a_2$, $C(x_3)=a_3$, $C(x_4)=a_1-a_2+a_3$}
    \label{coloring2}
  \end{center}
\end{figure}

 A Dehn $\Z$-coloring $C$ of $D$ is {\it trivial} if the following (i) or (ii) holds: (i) $\# \mathcal{C}(D, C)=1$; (ii) $\# \mathcal{C}(D, C)=2$ and $C(x_1)\not =C(x_2)$ for any two adjacent regions $x_1$ and $x_2$. 
A link $L$ is {\it Dehn $\Z$-colorable} if $L$ has a Dehn $\Z$-colored diagram $(D,C)$ such that $C$ is nontrivial. 
We note that a link $L$ is Dehn $\Z$-colorable if and only if ${\rm Det}(L)=0$, where ${\rm Det}(L)$ is an invariant of links called the {\it determinant} of $L$.

A transformation $C \mapsto C+t$ with $t \in \mathbb Z$ is called a {\it parallel transformation on ${\rm Col}_\mathbb{Z}(D)$}, where $C+t$ is a Dehn $\mathbb Z$-coloring of $D$ which maps a region $x$ to $C(x)+t$.

In this paper, we focus on the minimum number of colors.
\begin{definition} 
The {\it minimum number of Dehn $\Z$-colors} of a $\mathbb Z$-colorable link $L$ is the minimum number of distinct elements of $\Z$ which produce a nontrivially Dehn $\Z$-colored diagram of $L$, that is, 
\[
\min\Big\{\#\mathcal{C}(D,C) ~\Big|~ (D,C)\in \left\{
\begin{minipage}{4.3cm}
nontrivially Dehn $\Z$-colored\\
 diagrams of $L$
\end{minipage}
\right\}\Big\}.
\]
We denote it by $\mincol_\Z(L)$.
\end{definition}
We note that $\mincol_\Z(L)=2$ for a splittable link $L$, which is shown as follows.
A splittable link $L$ has a diagram $D=D_1 \sqcup D_2$ such that $D_1$ and $D_2$ are separated by a simple closed curve $\ell$ on $\mathbb R_2$. 
We then give the color $0$ to the regions appearing inside of $\ell$, and give the colors $0$ and $1$ to the regions appearing outside of $\ell$ so that $C(x_1)\not =C(x_2)$ for any two adjacent regions $x_1$ and $x_2$, and the unbounded region is colored by $0$. Such an assignment is a nontrivial Dehn $\mathbb Z$-coloring with two colors. Thus, $\mincol_\Z(L)=2$ for a splittable link $L$. 
Therefore, when we discuss the minimum number of Dehn $\Z$-colors, it is appropriate to restrict to ``nonsplittable'' $\mathbb Z$-colorable links.

\section{Simple Dehn $\mathbb Z$-colorings and a relationship with Fox $\mathbb Z$-colorings}

Fox colorings of link diagrams have been extensively studied for a long time. There is the correspondence, as illustrated in the upper picture of  Figure~\ref{coloring1}, from Dehn $\mathbb Z$-colorings to Fox $\mathbb Z$-colorings for a given diagram. 
In this correspondence, we distinguish colors of regions for Dehn $\Z$-colorings and those of arcs for Fox $\Z$-colorings by putting a bar over the colors of arcs, as shown in Figure~\ref{coloring1}. 
Then a trivial (resp. nontrivial) Dehn $\Z$-coloring corresponds to a trivial  (resp. nontrivial) Fox $\Z$-coloring as depicted in the lower picture of  Figure~\ref{coloring1}. We also note that a link $L$ is Dehn $\Z$-colorable if and only if $L$ is Fox $\Z$-colorable. 
Conversely, every time we fix a color assigned to the unbounded region of a given diagram, a reverse conversion from Fox $\mathbb Z$-colorings to Dehn $\mathbb Z$-colorings is given by the rule depicted in the upper picture of Figure~\ref{coloring3}. See also the lower picture of Figure~\ref{coloring3}, and the Dehn $\mathbb Z$-coloring is determined from the Fox $\mathbb Z$-coloring when we fix $0$ as the color of the unbounded region. 
\begin{figure}[ht]
  \begin{center}
    \includegraphics[clip,width=8.5cm]{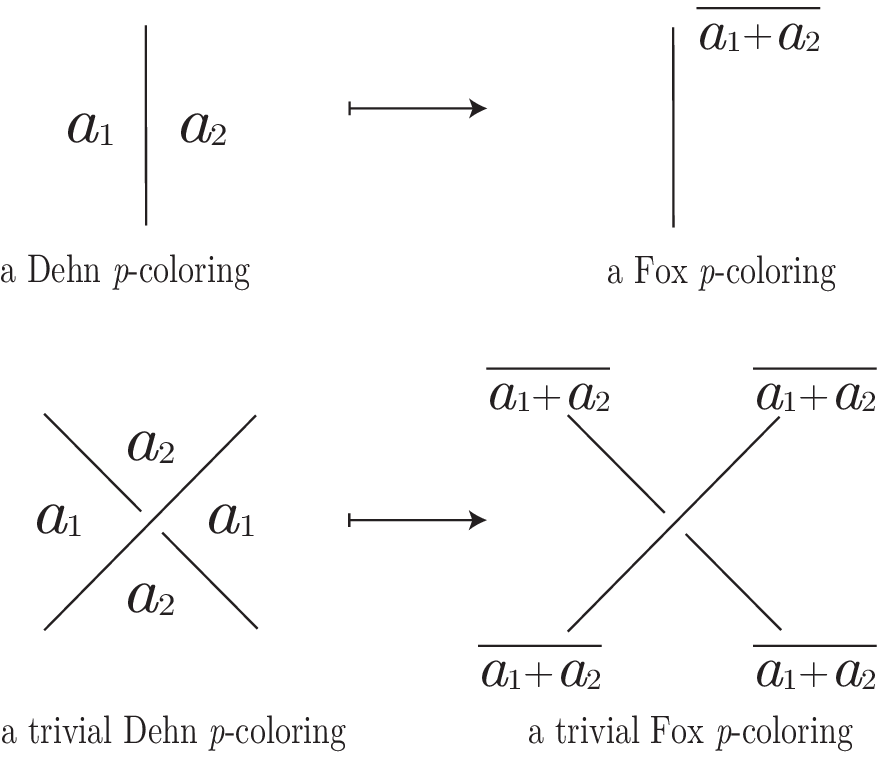}
    \caption{The conversion rule from Dehn colorings to Fox colorings}
    \label{coloring1}
  \end{center}
\end{figure}
\begin{figure}[ht]
  \begin{center}
    \includegraphics[clip,width=9cm]{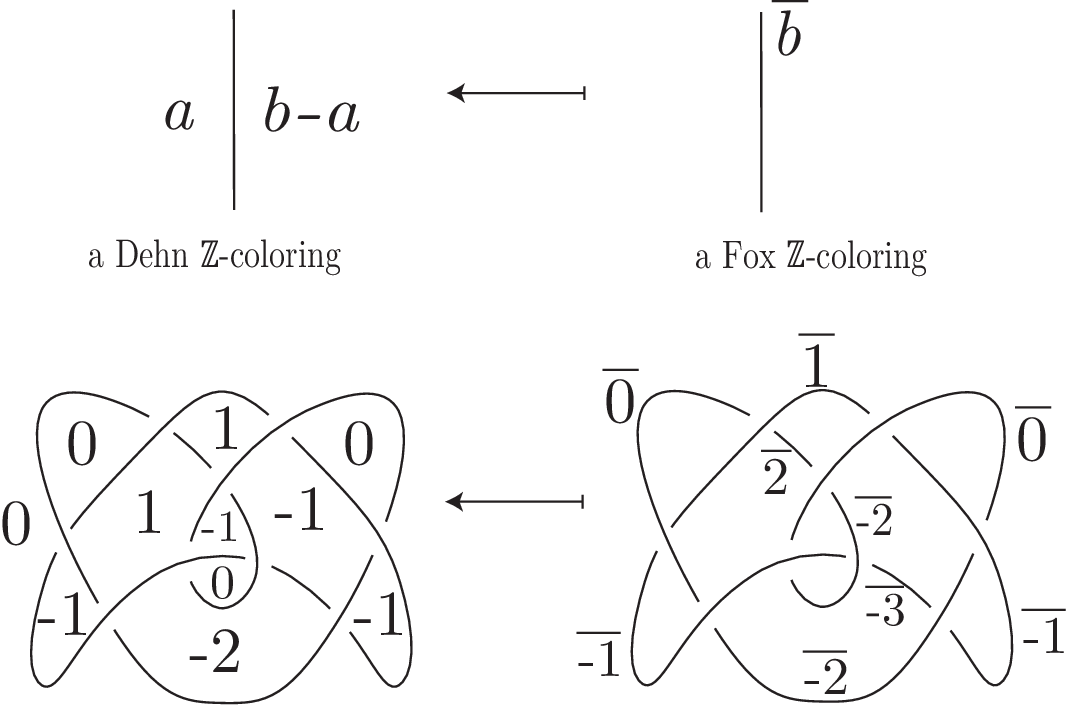}
    \caption{The conversion rule from Fox colorings to Dehn colorings}
    \label{coloring3}
  \end{center}
\end{figure}

Let $n$ be a non-negative integer.
For a crossing $c$ of a Dehn $\Z$-colored diagram $(D,C)$, we call $c$ an {\it $n$-diff crossing} if $|C(x_1)-C(x_4)|=|C(x_2)-C(x_3)|=n$, where $x_1, \ldots ,x_4$ are the regions that form $c$ as shown in Figure~\ref{coloring2}.
A Dehn $\Z$-coloring $C$ of a link diagram $D$ is {\it $n$-simple} if each crossing of $(D,C)$ is an $n$- or $0$-diff crossing, and there exists at least one $n$-diff crossing. 
We also simply say that a Dehn $\Z$-coloring $C$ is {\it simple} if it is $n$-simple for some $n$. 

On the other hand, the terms ``$n$-diff" and ``$n$-simple" were defined in the case of Fox $\Z$-colorings in \cite{ZhangJinDeng} and \cite{IchiharaMatsudo17}, respectively, as follows.
A crossing $c$ of $(D, \bar{C})$ is an {\it $n$-diff} crossing if the difference between the colors of the over-arc and an under-arc of $c$ is $n$.
A Fox $\mathbb Z$-coloring $\bar{C}$ of a diagram $D$ is {\it $n$-simple} if each crossing of $(D,\bar{C})$ is an $n$- or $0$-diff crossing, and there exists at least one $n$-diff crossing. 
For a Dehn $\mathbb Z$-colored diagram $(D,C)$,  
an $n$-diff crossing $c$ of $(D,C)$ corresponds to the $n$-diff one $c$ of the corresponding Fox $\Z$-colored diagram $(D,\bar{C})$ (see Figure~\ref{n-simple-crossings}).
An $n$-simple Dehn $\Z$-coloring $C$ of $D$ corresponds to the $n$-simple Fox $\Z$-coloring $\bar{C}$ of $D$.
\begin{figure}[ht]
  \begin{center}
    \includegraphics[clip,width=7.5cm]{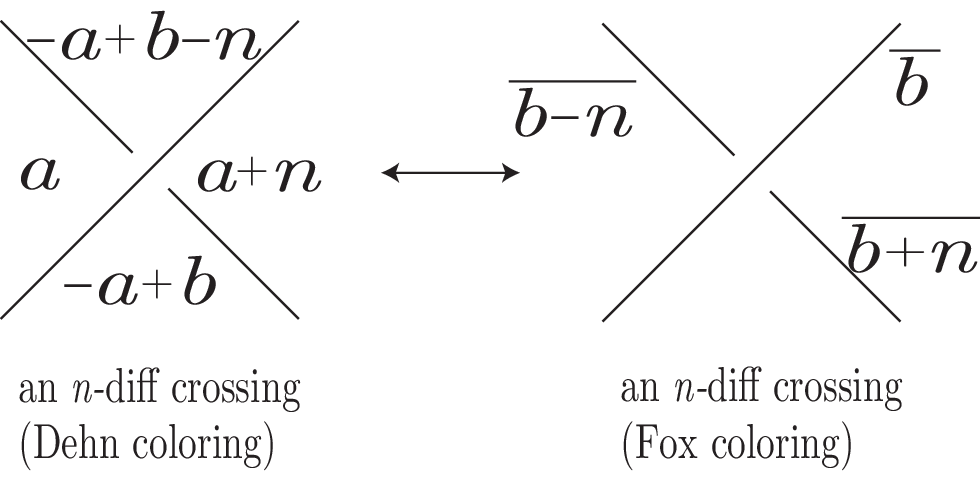}
    \caption{The correspondence for an $n$-diff crossing}
    \label{n-simple-crossings}
  \end{center}
\end{figure}

\begin{lemma}\label{lem:1-simple}
Any nonsplittable Dehn $\mathbb Z$-colorable link has a diagram admitting a $1$-simple Dehn $\mathbb Z$-coloring.
\end{lemma}
\begin{proof}
In \cite{Matsudo}, the same property was shown for nonsplittable Fox $\mathbb Z$-colorable links, that is,
any nonsplittable Fox $\mathbb Z$-colorable link has a diagram admitting a $1$-simple Fox $\mathbb Z$-coloring.

Let $L$ be any nonsplittable Dehn $\mathbb Z$-colorable link, which is also Fox $\mathbb Z$-colorable.
By the property shown in \cite{Matsudo}, we have a diagram $D$ of $L$ and a $1$-simple Fox $\mathbb Z$-coloring $\bar{C}$ of $D$. 
Let $C$ be the Dehn $\mathbb Z$-coloring of $D$, with the color $0$ for the unbounded region of $D$, 
corresponding to $\bar{C}$ by the rule depicted in Figure~\ref{coloring3}.
Then, as shown in Figure~\ref{n-simple-crossings}, since any $1$- (or $0$-) diff crossing $c$ of $(D, \bar{C})$ corresponds to the $1$- (or $0$-) diff crossing $c$ of $(D, C)$, $C$ is $1$-simple. 
\end{proof}

\section{Results and proofs} \label{}
In this section, we prove Theorem~\ref{main_theorem}.

\begin{lemma}\label{lemma:colors>2}
For any nonsplittable Dehn $\mathbb Z$-colorable link $L$ and any nontrivially Dehn $\Z$-colored diagram $(D,C)$ of $L$, $\#\C(D,C)\geq 3$.
\end{lemma}
\begin{proof}
Let $L$ be a nonsplittable Dehn $\mathbb{Z}$-colorable link, and let $(D, C)$ be a nontrivially Dehn $\mathbb{Z}$-colored diagram of $L$. Suppose, for the sake of contradiction, that $\#C(D, C) < 3$.
Since the coloring $C$ is nontrivial, $C$ must use at least two distinct colors. Hence, we may assume that $\#C(D, C) = 2$.
By applying a parallel transformation
if necessary, we can assume without loss of generality that the two colors are $0$ and $a$ for some non-zero integer $a \in \mathbb{Z}$.
Then, the coloring patterns of adjacent regions are as depicted in Figure~\ref{colors-than3}.
This implies that the corresponding Fox $\mathbb Z$-colored diagram $(D, \bar{C})$ is colored with at most three colors, namely $\bar{0}$, $\bar{a}$ and $\overline{2a}$, as shown in Figure~\ref{colors-than3}. 
However, this contradicts the result of \cite{IchiharaMatsudo16}, which states that the minimum number of Fox $\mathbb Z$-colors of $L$ is at least four. Consequently, we conclude that $\#C(D,C)\geq 3$.

\begin{figure}[ht]
  \begin{center}    \includegraphics[clip,width=6cm]{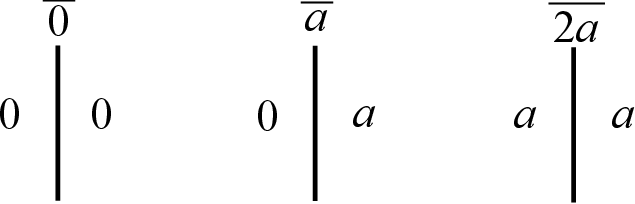}
    \caption{Coloring patterns for adjacent regions}
    \label{colors-than3}
  \end{center}
\end{figure}

\end{proof}

\begin{lemma}\label{lemma:colors=3}
For any nonsplittable Dehn $\mathbb Z$-colorable link $L$, there exists a nontrivially Dehn $\Z$-colored diagram $(D,C)$ of $L$ with  $\#\C(D,C)\leq 3$.
\end{lemma}

\begin{proof}
Let $(D,C)$ be a nontrivially Dehn $\Z$-colored diagram of $L$ such that $C$ is $1$-simple (see Lamme~\ref{lem:1-simple}). 
Here, by applying a parallel transformation, we may assume that the minimum color of the Dehn $\mathbb{Z}$-coloring $C$ is $0$.
Let $M$ be the maximum color on $(D,C)$, and we assume that $M\geq 3$. 
We first construct a nontrivially Dehn $\mathbb{Z}$-colored diagram $(D', C')$ with $\mathcal{C}(D',C') \subset\{0, \ldots , M-1\}$ by removing all regions colored $M$ from $(D, C)$ according to the following procedure.
\begin{itemize}
    \item(Step 1) Remove crossings colored with four $M$s.
    \item(Step 2) Remove crossings colored with two $M$s.
    \item(Step 3) Remove $n$-gons colored $M$.
\end{itemize}

In this proof, suppose that $a$ is a color with $0\leq a<M$, and additional conditions may be imposed on $a$ depending on the situation.

(Step 1) 
First, we remove crossings whose four adjacent regions are all colored $M$ by the move depicted in Figure~\ref{onlyM}.
We note that the resulting $\Z$-coloring is $1$-simple.

\begin{figure}[ht]
    \centering
    \includegraphics[width=0.9\linewidth]{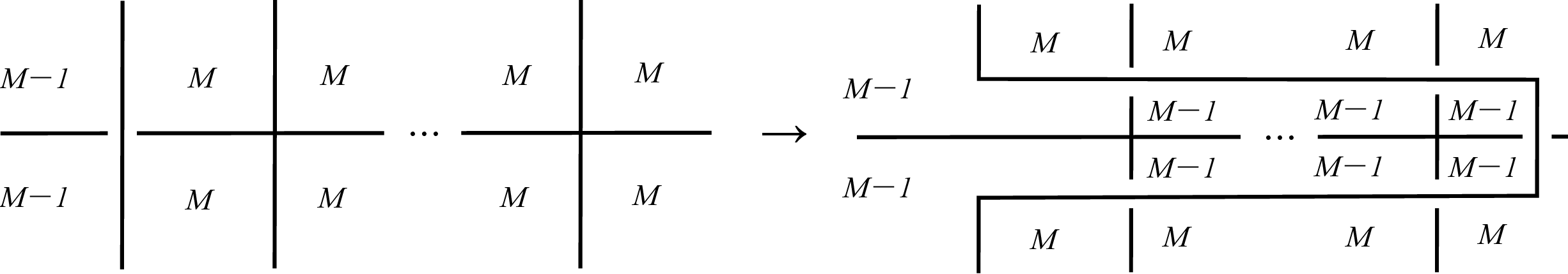}
    \caption{Removal of a crossing colored only $M$}
    \label{onlyM}
\end{figure}

(Step 2)
After Step 1, there are no crossings whose four adjacent regions are all colored $M$.
We now focus on crossings that have exactly two regions colored $M$.
Such crossings fall into two cases: either the two regions colored $M$ are adjacent, or they are diagonally opposite.
First, in the case where the regions colored $M$ are adjacent, we eliminate such crossings by applying the deformation depicted in Figure~\ref{2Ms-cons.}.
\begin{figure}[ht]
    \centering
    \includegraphics[width=0.6\linewidth]{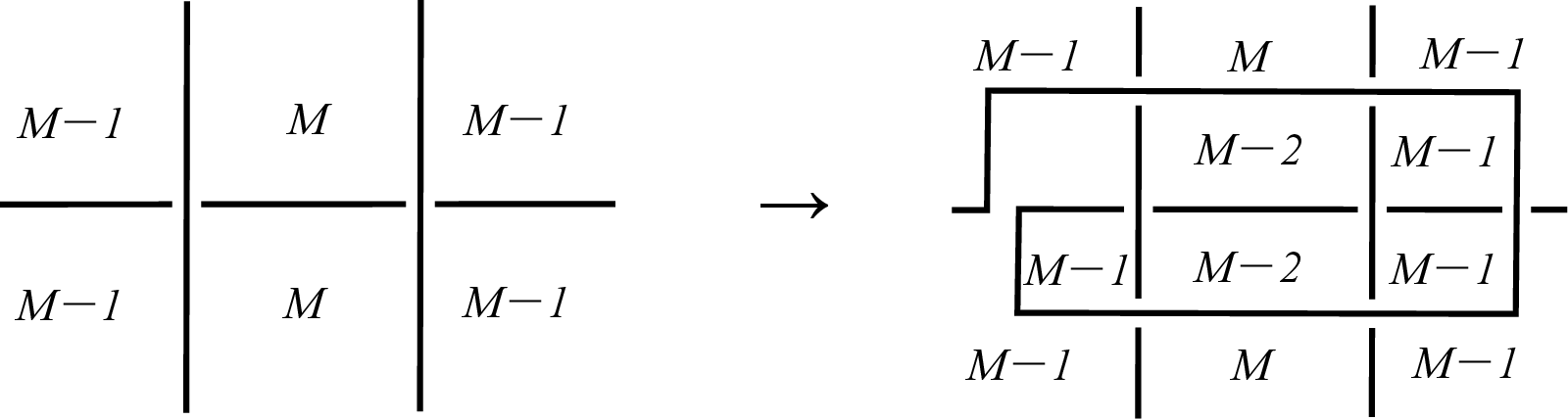}
    \caption{Removal of an arc between two regions colored $M$}
    \label{2Ms-cons.}
\end{figure}

Next, in the case where $M$ appears in diagonally opposite regions, we eliminate such crossings by applying the deformations shown in Figures~\ref{2Ms-opp.1} and \ref{2Ms-opp.2}, depending on the surrounding local configuration.

\begin{figure}[ht]
    \centering
 \includegraphics[width=0.7\linewidth]{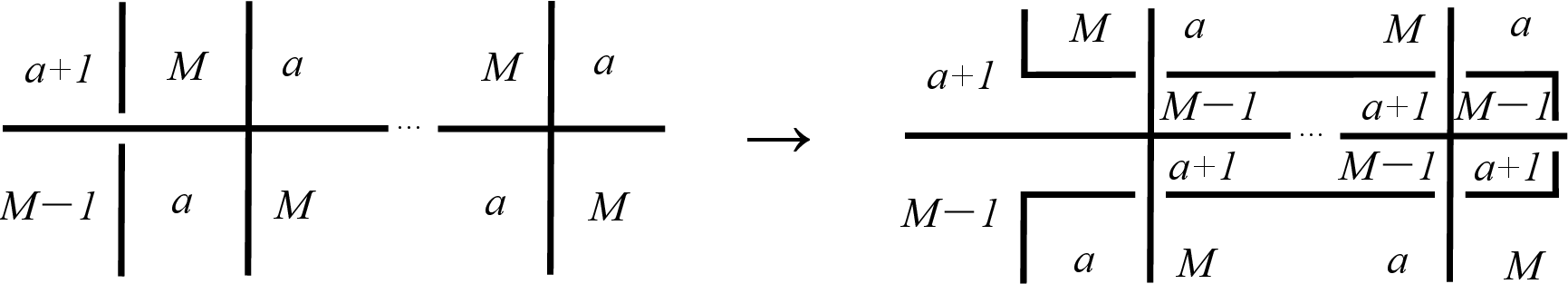}
    \caption{Removal of a crossing with $M$ in diagonally opposite regions}
    \label{2Ms-opp.1}
\end{figure}

\begin{figure}[ht]
    \centering    \includegraphics[width=0.7\linewidth]{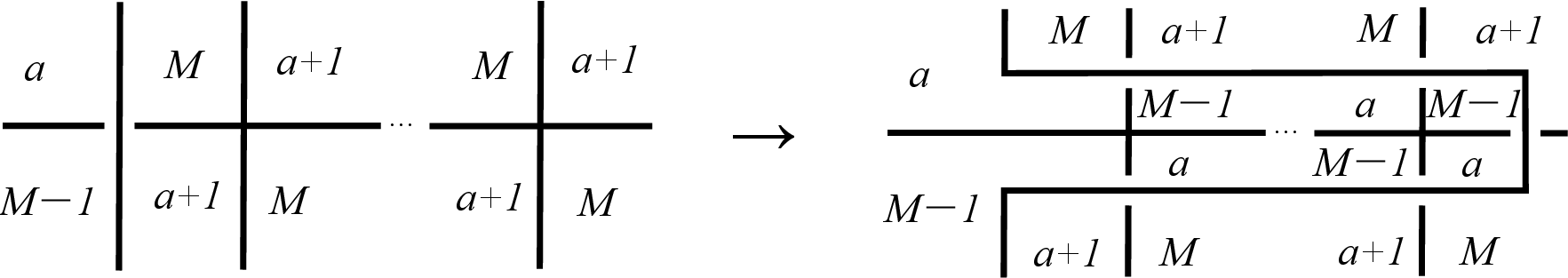}
    \caption{Removal of a crossing with $M$ in diagonally opposite regions}
    \label{2Ms-opp.2}
\end{figure}

We note that the resulting $\Z$-coloring is $1$-simple.

(Step 3)
We eliminate $n$-gons colored $M$ as follows. Due to the 1-simple Dehn $\mathbb{Z}$-coloring condition, $n$ must be an even number. Furthermore, according to Steps 1 and 2, no other regions colored $M$ occur in the vicinity of an $n$-gon colored $M$.
First, suppose that $n>4$. Then any $n$-gon colored $M$ admits a local configuration as illustrated on the left-hand side of Figures~\ref{into4gon-4}, \ref{into4gon-1}, or \ref{into4gon-2}. For each case, the $n$-gon can be separated into a $4$-gon  and $(n-2)$-gon via the transformations illustrated in Figures~\ref{into4gon-4}, \ref{into4gon-1}, or \ref{into4gon-2}.

\begin{figure}[ht]
    \centering    \includegraphics[width=0.7\linewidth]{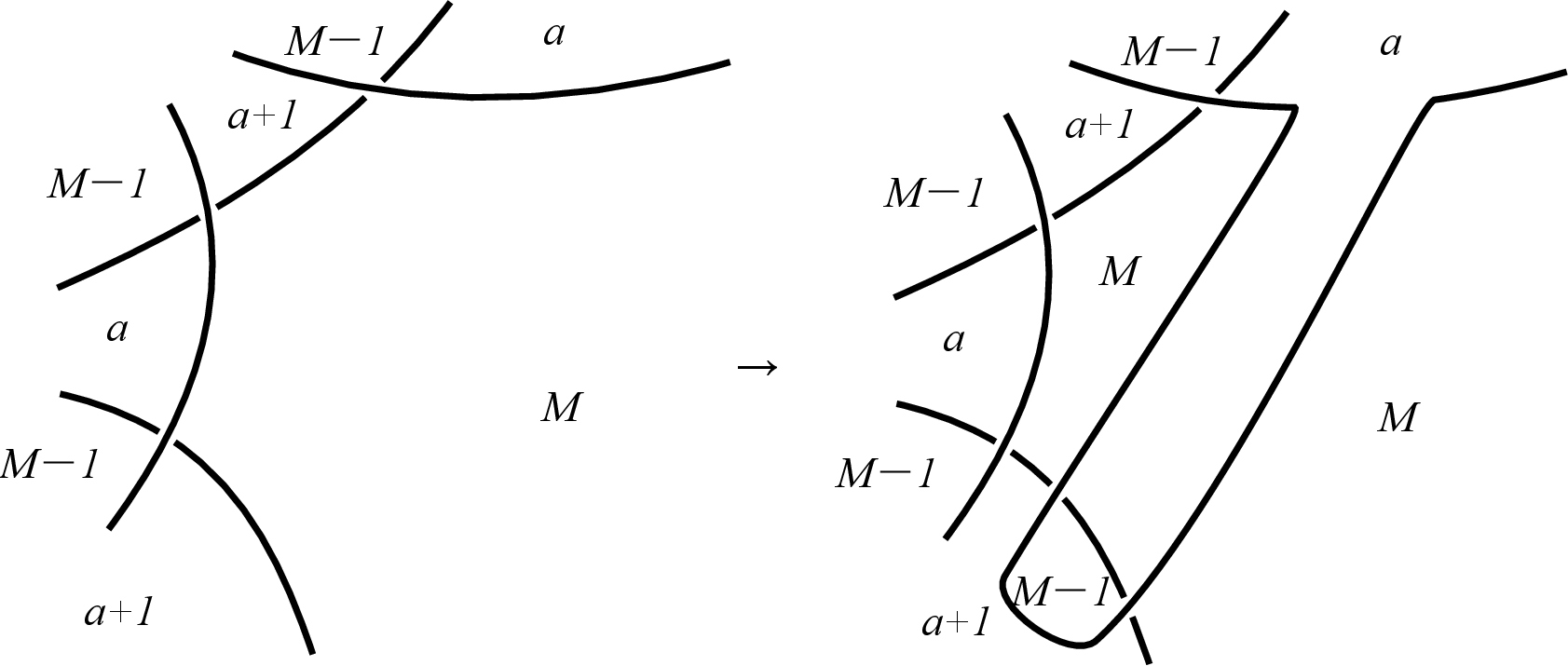}
    \caption{Decomposition of an $n$-gon into a 4-gon and an $(n-2)$-gon}
    \label{into4gon-4}
\end{figure}
\begin{figure}[ht]
    \centering    \includegraphics[width=0.7\linewidth]{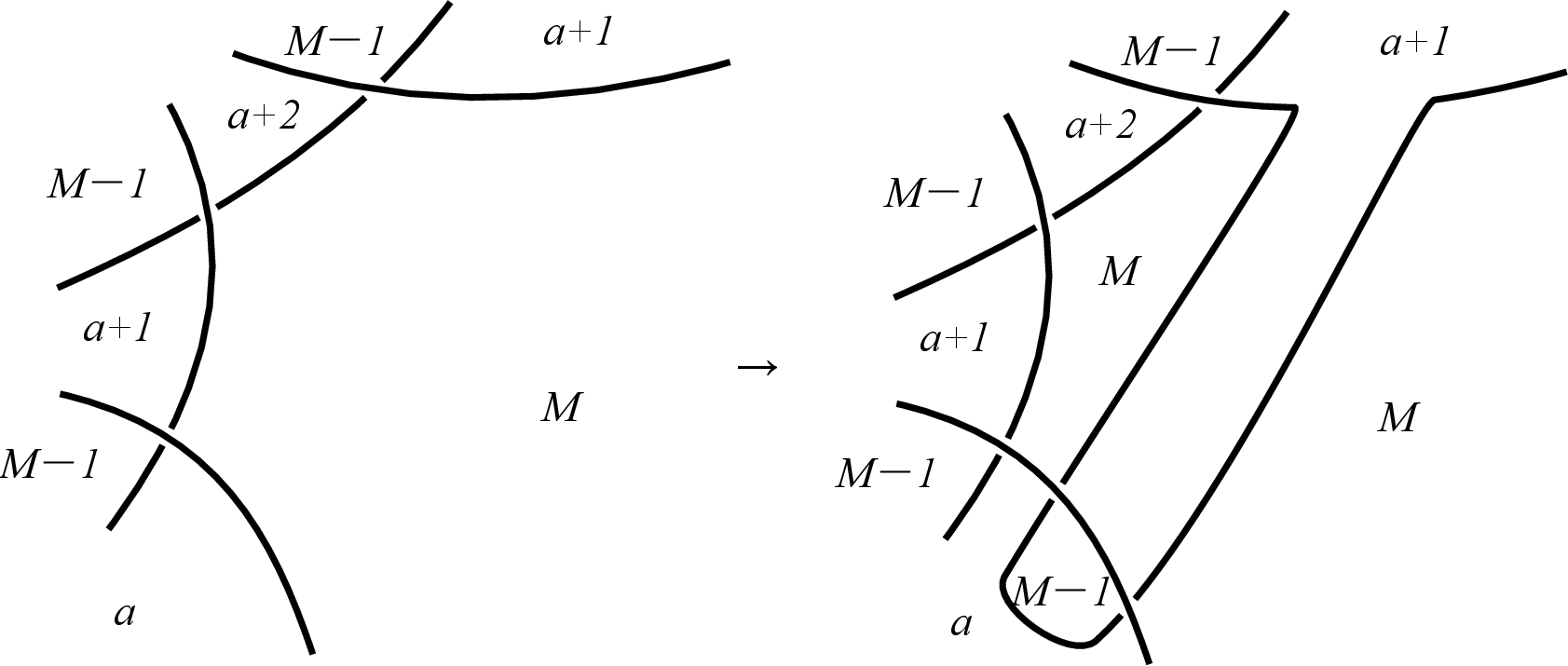}
    \caption{Decomposition of an $n$-gon into a 4-gon and an $(n-2)$-gon}
    \label{into4gon-1}
\end{figure}
\begin{figure}[ht]
    \centering    \includegraphics[width=0.7\linewidth]{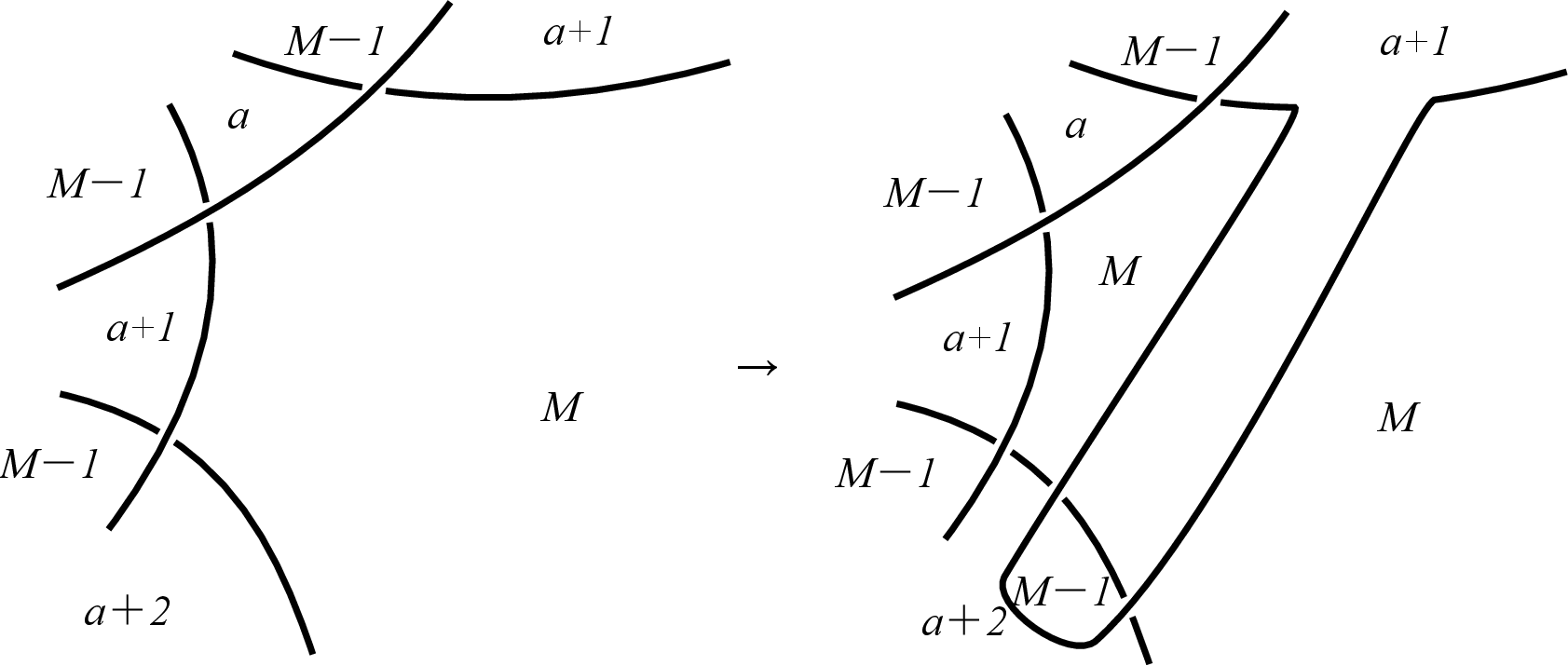}
    \caption{Decomposition of an $n$-gon into a 4-gon and an $(n-2)$-gon}
    \label{into4gon-2}
\end{figure}
By repeating this operation, all $n$-gons colored $M$ can be separated into 4-gons, and the remaining regions colored $M$ are 4-gons and $2$-gons which have no other regions colored $M$ in their vicinity. 
We note that the resulting $\Z$-coloring is $1$-simple.

Next, we eliminate the remaining $4$-gons colored $M$.
Such 4-gons are reduced to the cases shown on the left-hand sides of Figures \ref{4gon-case1-1}, \ref{4gon-case1-2} and \ref{4gon-case2}.
Note that for the case depicted in Figures \ref{4gon-case1-1}, we assume that $a\not= 0$.  
For each case, the $4$-gon colored $M$ can be eliminated by applying the transformation shown in the corresponding figure.
\begin{figure}[ht]
    \centering    \includegraphics[width=0.7\linewidth]{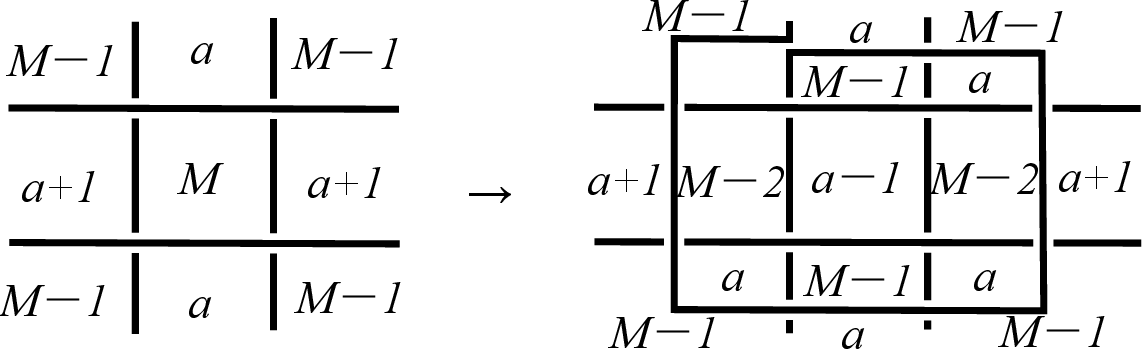}
    \caption{Removal of a $4$-gon colored with $M$ ($a\not= 0$)}
    \label{4gon-case1-1}
\end{figure}
\begin{figure}[ht]
    \centering    \includegraphics[width=0.7\linewidth]{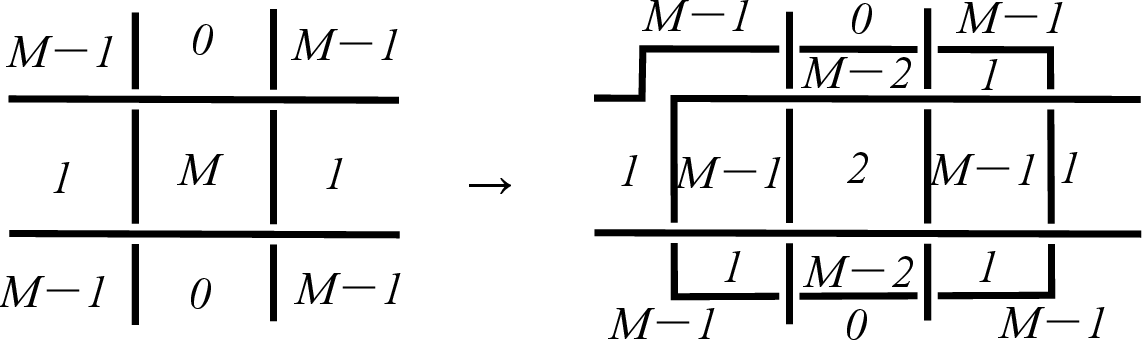}
    \caption{Removal of a $4$-gon colored with $M$}
    \label{4gon-case1-2}
\end{figure}
\begin{figure}[ht]
    \centering    \includegraphics[width=0.7\linewidth]{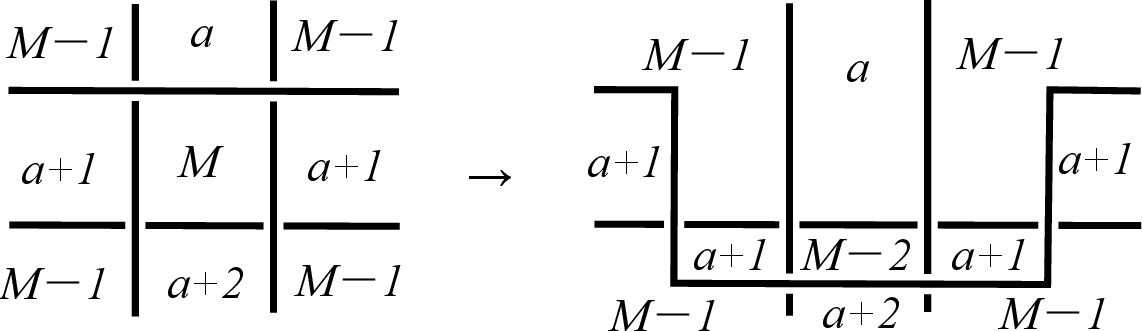}
    \caption{Removal of a $4$-gon colored with $M$}
    \label{4gon-case2}
\end{figure}

Lastly, we eliminate the remaining $2$-gons colored $M$.
Such $2$-gons are shown on the left-hand side of Figure~\ref{2gon}, and they can be eliminated by applying the move depicted in Figure~\ref{2gon}.

\begin{figure}[ht]
    \centering    \includegraphics[width=0.5\linewidth]{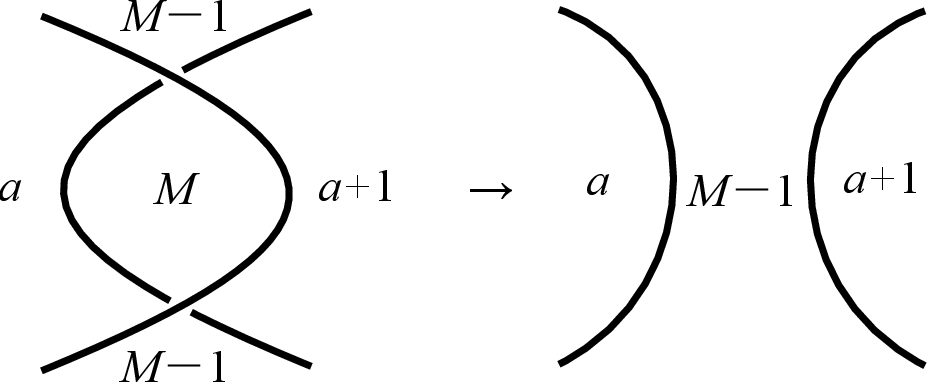}
    \caption{Removal of a $2$-gon colored with $M$}
    \label{2gon}
\end{figure}

We note that the resulting $\Z$-coloring is $1$-simple.

After carrying out all of the steps above, no region colored $M$ remains in the diagram, and for the resulting colored diagram $(D',C')$, we have $\C(D',C') \subset \{0,1,\ldots , M-1\}$. 
We note that, under the condition $M\geq 3$, no negative integers arise in any of the above transformations.

We again denote the newly obtained $(D',C')$ by $(D,C)$, and apply a parallel transformation for $(D,C)$, if needed, so that its minimum color is 0.
Repeat the above three steps for $(D,C)$ while the maximum color is greater than $2$. 

Consequently, we have a nontrivially Dehn $\mathbb Z$-colored diagram $(D,C)$ of $L$ with $\C(D,C)\subset \{0,1,2\}$, i.e. $\#\C(D,C)\leq 3$. 
\end{proof}

\begin{proof}[Proof of Theorem~\ref{main_theorem}.]
This follows from Lemmas~\ref{lemma:colors>2} and \ref{lemma:colors=3}.
\end{proof}

\section*{Acknowledgments}
The second author was supported by JSPS KAKENHI Grant Number JP26K06788.

\end{document}